\documentclass[12pt]{article}
\usepackage{amssymb}
\usepackage{amsmath,amsthm}
\usepackage[latin1]{inputenc}
\usepackage{graphicx}
\usepackage{hyperref}
\usepackage{enumerate}
\usepackage{tikz}

\hypersetup{colorlinks=true, linkcolor=blue, citecolor=blue, urlcolor=blue}

\newtheorem{remark}{Remark}

\newtheorem{theorem}[remark]{Theorem}
\newtheorem{proposition}[remark]{Proposition}
\newtheorem{corollary}[remark]{Corollary}

\newcommand{\ter}{\operatorname{ter}}

\title{On the complexity of computing the $k$-metric dimension of graphs}

\author{Ismael G. Yero$^{(1)}$, Alejandro Estrada-Moreno$^{(2)}$\\ and Juan A.
Rodr\'{\i}guez-Vel\'{a}zquez$^{(2)}$
    \\
    $^{(1)}${\small Departamento de Matem\'aticas, Escuela Polit\'ecnica Superior de Algeciras}\\
{\small Universidad de C\'adiz,} {\small
Av. Ram\'on Puyol s/n, 11202 Algeciras, Spain.} \\ {\small
ismael.gonzalez\@@uca.es}\\
$^{(2)}${\small Departament d'Enginyeria Inform\`atica i Matem\`atiques,}\\
{\small Universitat Rovira i Virgili,}  {\small Av. Pa\"{\i}sos
Catalans 26, 43007 Tarragona, Spain.} \\{\small
alejandro.estrada\@@urv.cat, juanalberto.rodriguez\@@urv.cat}
}

\begin{document}

\maketitle

\begin{abstract}
Given a connected graph $G=(V,E)$, a set $S\subseteq V$ is a $k$-metric generator for $G$ if for any two different vertices $u,v\in V$, there exist at least $k$ vertices $w_1,...,w_k\in S$ such that $d_G(u,w_i)\ne d_G(v,w_i)$ for every $i\in \{1,...,k\}$. A metric generator of minimum cardinality is called a $k$-metric basis and its cardinality the $k$-metric dimension of $G$. We study some problems regarding the complexity of some $k$-metric dimension problems. For instance, we show that the problem of computing the $k$-metric dimension of graphs is $NP$-Complete. However, the problem is solved in linear time for the particular case of trees.
\end{abstract}

{\it Keywords:} $k$-metric dimension; $k$-metric dimensional graph; metric dimension; $NP$-complete problem; Graph algorithms.

{\it AMS Subject Classification Numbers:}   68Q17; 05C05; 05C12;	05C85.

\section{Introduction}

Let $\mathbb{R}_{\geq 0}$ denote  the set of nonnegative real numbers.
A \emph{metric space} is a pair $(X,d)$, where $X$ is a set of points and $d:X\times X\to \mathbb{R}_{\geq 0}$ satisfies
$d(x,y)=0$ if and only if $x=y$, $d(x,y)=d(y,x)$  for all $x,y \in X$ and $d(x,y)\le d(x,z)+d(z,y)$  for all $x,y,z\in X$.
A \emph{generator} of a metric space $(X,d)$ is a set $S$ of points in the space  with the property that every point of the space is uniquely determined by the distances from the elements of $S$.
A point $v\in X$ is said to \emph{distinguish} two points $x$ and $y$ of $X$ if $d(v,x)\ne d(v,y)$.
Hence, $S$ is a generator if and only if any pair of points of $X$ is distinguished by some element of $S$.

Let $\mathbb{N}$ denote the set of nonnegative integers.
Given a connected graph $G=(V,E)$, we consider the function $d_G:V\times V\rightarrow \mathbb{N}$, where $d_G(x,y)$ is the length of a shortest path between $u$ and $v$. Clearly, $(V,d_G)$ is a metric space.
% , i.e., $d_G$ satisfies
%  $d_G(x,x)=0$  for all $x\in V$, $d_G(x,y)=d_G(y,x)$  for all $x,y \in V$ and $d_G(x,y)\le d_G(x,z)+d_G(z,y)$  for all $x,y,z\in V$.
% A vertex $v\in V$ is said to distinguish two vertices $x$ and $y$ if $d_G(v,x)\ne d_G(v,y)$.
The diameter of a graph is understood in this metric.
%Alternatively, the diameter can be defined via the notion of eccentricity of a vertex, which is defined as $\varepsilon(v)=\sup\{d_G(v,u): u\in V-\{v\}\}$.
%Namely, $\diam(G)=\max\{\varepsilon(v): v\in V\}$.
%Similarly, the \emph{radius} of a graph is defined as $r(G)=\min \{\varepsilon(v): v\in V\}$.

A vertex set $S\subseteq V$ is said to be a \emph{metric generator} for $G$ if it is a generator of the metric space $(V,d_G)$.
% any pair of vertices of $G$ is
% distinguished by some element of $S$.
A minimum metric generator is called a \emph{metric basis}, and
its cardinality the \emph{metric dimension} of $G$, denoted by $\dim(G)$.
Motivated by some problems regarding unique location of intruders in a network, the concept of metric
dimension of a graph was introduced by Slater in \cite{Slater1975}, where the metric generators were called \emph{locating sets}. The concept of metric dimension of a graph was also introduced by Harary and Melter in \cite{Harary1976}, where metric generators were called \emph{resolving sets}. Applications
of this invariant to the navigation of robots in networks are discussed in \cite{Khuller1996} and applications to chemistry in \cite{Johnson1993,Johnson1998}.  This graph parameter was studied further in a number
of other papers including, for instance \cite{Bailey2011, Caceres2007, Chartrand2000, Kuziak2013, Melter1984, Yero2011}.
Several variations of metric generators including resolving dominating sets \cite{Brigham2003}, locating dominating sets \cite{Rall1984}, independent resolving sets \cite{Chartrand2003}, local metric sets \cite{Okamoto2010}, strong resolving sets \cite{Sebo2004}, etc. have been introduced and studied.
%%%%%%%%%%%%%%%%%%%%%

From now on we consider an extension of the concept of metric generators introduced in \cite{Estrada-Moreno2013} by the authors of this paper. Given a simple and connected graph $G=(V,E)$, a set $S\subseteq V$ is said to be a \emph{$k$-metric generator} for $G$ if and only if any pair of vertices of $G$ is distinguished by at least $k$ elements of $S$, {\em i.e.}, for any pair of different vertices $u,v\in V$, there exist at least $k$ vertices $w_1,w_2,...,w_k\in S$ such that
\begin{equation}\label{conditionDistinguish}
d_G(u,w_i)\ne d_G(v,w_i),\; \mbox{\rm for every}\; i\in \{1,...,k\}.
\end{equation}
A $k$-metric generator of minimum cardinality in $G$ is  called a \emph{$k$-metric basis} and its cardinality the \emph{$k$-metric dimension} of $G$, which is denoted by $\dim_{k}(G)$, \cite{Estrada-Moreno2013}. Note that every $k$-metric generator $S$ satisfies that $|S|\geq k$ and, if $k>1$, then $S$ is also a $(k-1)$-metric generator. Moreover,   $1$-metric generators are the standard metric generators (resolving sets or locating sets as defined in \cite{Harary1976} or \cite{Slater1975}, respectively). Notice that if $k=1$, then the problem of checking if a set $S$ is a metric generator reduces to check condition (\ref{conditionDistinguish}) only for those vertices $u,v\in V- S$, as every vertex in $S$ is distinguished at least by itself. Also, if $k=2$, then condition (\ref{conditionDistinguish}) must be checked only for those pairs having at most one vertex in $S$, since two vertices of $S$ are distinguished at least by themselves. Nevertheless, if $k\ge 3$, then condition (\ref{conditionDistinguish}) must be checked for every pair of different vertices of the graph.

In this article we show the NP-Hardness of the problem of computing the $k$-metric dimension of graphs. To do so, we first prove that the decision problem regarding whether $\dim_{k}(G)\le r$ for some graph $G$ and some integer $r\ge k+1$ is $NP$-complete. The particular case of trees is separately addressed, based on the fact that for trees, the problem mentioned above becomes polynomial. We say that a connected graph $G$ is \emph{$k$-metric dimensional} if $k$ is the largest integer such that there exists a $k$-metric basis for $G$.   We also show that
 the problem of finding the integer $k$ such that a graph $G$  is $k$-metric dimensional can be solved in polynomial time.
The reader is referred to \cite{Estrada-Moreno2013} for combinatorial results on the $k$-metric dimension,  including
tight bounds and some closed formulae. The article is organized as follows. In Section \ref{k-metric-dimensional}
we analyze the problem of computing the largest integer $k$ such that there exists a $k$-metric basis. In Section \ref{k-metric-dimension-problem} we show that the decision problem regarding whether the $k$-metric dimension of a graph does not exceed a positive integer is $NP$-complete, which gives also the NP-Hardness of  computing $\dim_{k}(G)$ for any graph $G$. The procedure of such proof is done by using some similar techniques like those ones already presented in \cite{Khuller1996} while studying the computational complexity problems related to the standard metric dimension of graphs.
Finally, in Section \ref{k-metric-dimension-tree} we give an algorithm for determining the value of $k$ such that a tree is $k$-metric dimensional and present two algorithms for computing the $k$-metric dimension and obtaining a $k$-metric basis of any tree. We also show that all algorithms presented in this section run in linear time.

Throughout the paper, we use the notation $S_{1,n-1}$, $C_n$ and $P_n$ for star graph, cycle graphs and path graphs of order $n$, respectively. For a vertex $v$ of a graph $G$, $N_G(v)$ denote the set of neighbors or \emph{open neighborhood} of $v$ in $G$. The \emph{closed neighborhood}, denoted by $N_G[v]$, equals $N_G(v) \cup \{v\}$. If there is no ambiguity, we simple write $N(v)$ or $N[v]$. We also refer to the degree of $v$ as $\delta(v)=|N(v)|$. The minimum and maximum degrees of $G$ are denoted by $\delta(G)$ and $\Delta(G)$, respectively.

In this work  the remain definitions are given the first time that the concept is found in the text.

\section{$k$-metric dimensional graphs.}\label{k-metric-dimensional}

Notice that if a graph $G$ is a $k$-metric dimensional graph, then for every positive integer $k'\le k$, $G$ has at least  a $k'$-metric basis. Since for every pair of vertices $x,y$ of a graph $G$ we have that they are distinguished at least by themselves, it follows that the whole vertex set $V(G)$ is a $2$-metric generator for $G$ and, as a consequence it follows that every graph $G$ is $k$-metric dimensional for some $k\ge 2$. On the other hand, for any connected graph $G$ of order $n>2$ there exists at least one vertex $v\in V(G)$ such that $\delta(v)\ge 2$. Since $v$ does not distinguish any pair $x,y\in N_G(v)$, there is no $n$-metric dimensional graph of order $n>2$.

\begin{remark}\label{remarkKMetric}
Let $G$ be a $k$-metric dimensional graph of order $n$. If $n\ge 3$ then, $2\le k\le n-1$. Moreover, $G$ is $n$-metric dimensional if and only if $G\cong K_2$.
\end{remark}

Next we present a characterization of $k$-metric dimensional graphs already known from \cite{Estrada-Moreno2013}. To this end, we need  some additional terminology. Given two vertices $x,y\in V(G)$, we say that the set of \textit{distinctive vertices} of $x,y$ is $${\cal D}_G(x,y)=\{z\in V(G): d_{G}(x,z)\ne d_{G}(y,z)\}.$$

%We say that two vertices $x,y\in V(G)$ are \emph{$k$-nearly twin} if  $\vert{\cal D}_G^*(x,y)\vert=k$.  We also say that a vertex $x$ is a \emph{$k$-nearly twin} if there exists another vertex $y$ such that  $x$ and $y$ are $k$-nearly twin vertices. Notice that two vertices $x,y$ are twin vertices if and only if $x,y$ are $0$-nearly twin vertices, justifying the terminology used.

\begin{theorem}{\em \cite{Estrada-Moreno2013}}\label{theokmetric}
A graph  $G$ is $k$-metric dimensional if and only if $k=\displaystyle\min_{x,y\in V(G)}\vert {\cal D}_G(x,y)\vert .$
\end{theorem}

Now we consider the problem of finding the integer $k$ such that a graph $G$ of order $n$ is $k$-metric dimensional.

$$\begin{tabular}{|l|}
  \hline
  \mbox{$k$-DIMENSIONAL GRAPH PROBLEM}\\
  \mbox{INSTANCE: A connected graph $G$ of order $n\ge 3$}\\
  \mbox{PROBLEM: Find the integer $k$, $2\le k\le n-1$, such that $G$ is $k$-metric dimensional}\\
  \hline
\end{tabular}$$

\begin{remark}\label{remark_K-Dimensional}
Let $G$ be a connected graph of order $n\geq 3$. The time complexity of computing the value $k$ for which $G$ is $k$-metric dimensional is $O(n^3)$.
\end{remark}

\begin{proof}
We can initially compute the distance matrix $\mathrm{DistM}_G$, by using the well-known Floyd-Warshall algorithm, whose time complexity is $O(n^3)$. The distance matrix $\mathrm{DistM}_G$ is symmetric of order $n\times n$ whose rows and columns are labeled by vertices, with entries between $0$ and $n-1$. Now observe that $z\in\mathcal{D}_G(x,y)$ if and only if $\mathrm{DistM}_G(x,z)\ne \mathrm{DistM}_G(y,z)$.

Given the matrix $\mathrm{DistM}_G$, the process of computing how many vertices belong to $\mathcal{D}_G(x,y)$ for each of the $\displaystyle\binom{n}{2}$ pairs $x,y\in V(G)$ can be checked in linear time. Therefore, the overall running time of such a process is bounded by the cubic time of the Floyd-Warshall algorithm.
\end{proof}

\section{The $k$-metric dimension problem}\label{k-metric-dimension-problem}

Since the problem of computing the value $k'$ for which a given graph is $k'$-metric dimensional is polynomial, we can study the problem of deciding whether the $k$-metric dimension, $k\le k'$, of $G$ is less than or equal to $r$, for some $r\ge k+1$, {\em i.e.}, the following decision problem.

$$\begin{tabular}{|l|}
  \hline
  \mbox{$k$-METRIC DIMENSION PROBLEM}\\
  \mbox{INSTANCE: A $k'$-metric dimensional graph $G$ of order $n\ge 3$ and integers $k,r$}\\
  \hspace*{2.4cm}\mbox{ such that $1\le k\le k'$ and $k+1\le r\le n$.}\\
  \mbox{QUESTION: Is $\dim_k(G)\le r$?}\\
  \hline
\end{tabular}$$\\

Next we prove that $k$-METRIC DIMENSION PROBLEM is $NP$-complete. The proof that the $k$-METRIC DIMENSION PROBLEM is $NP$-complete for $k=1$ was given by Khuller \emph{et al}. \cite{Khuller1996}. As a kind of generalization of the technique used in \cite{Khuller1996} for $k=1$, we also use a reduction from $3$-SAT in order to analyze the NP-completeness of the $k$-METRIC DIMENSION PROBLEM.

Our problem is clearly in $NP$, since verifying that a given subset $S\subseteq V(G)$ with $k+1\le |S|\le r$ is a $k$-metric generator for a graph $G$, can be done in polynomial time by using some similar procedure like that described in the proof of Remark \ref{remark_K-Dimensional}. In order to present the reduction from $3$-SAT, we need some terminology and notation. From now on, we assume $x_1,...,x_n$ are the variables; $Q_1,...,Q_s$ are the clauses; and $x_1,\overline{x_1},x_2,\overline{x_2},...,x_n,\overline{x_n}$ are the literals, where $x_i$ represents a positive literal of the variable while $\overline{x_i}$ represents a negative literal.

We consider an arbitrary input to $3$-SAT, that is, a boolean formula $\mathcal{F}$ with $n$ variables and $s$ clauses. In this reduction, without loss of generality, we assume that the formula $\mathcal{F}$ has $n\ge 4$ variables. Let $X=\{x_1,x_2,...,x_n\}$ be the set of variables and let $\mathcal{Q}=\{Q_1,Q_2,...,Q_s\}$ be the set of clauses. Now we construct a graph $G_F$ in the following way.
\begin{itemize}
\item For every $x_i\in X$, we take an even cycle $C^i$ of order $4\left\lceil\frac{k}{2}\right\rceil+2$ and we denote by $F_i$ (the false node) and by $T_i$ (the true node) two diametral vertices of $C^i$. Then we denote by $f_i^1,f_i^2,...,f_i^{2\left\lceil\frac{k}{2}\right\rceil}$ the half vertices of $C^i$ closest to $F_i$ and we denote by $t_i^1,t_i^2,...,t_i^{2\left\lceil\frac{k}{2}\right\rceil}$ the half vertices of $C^i$ closest to $T_i$ (see Figure \ref{variable}).
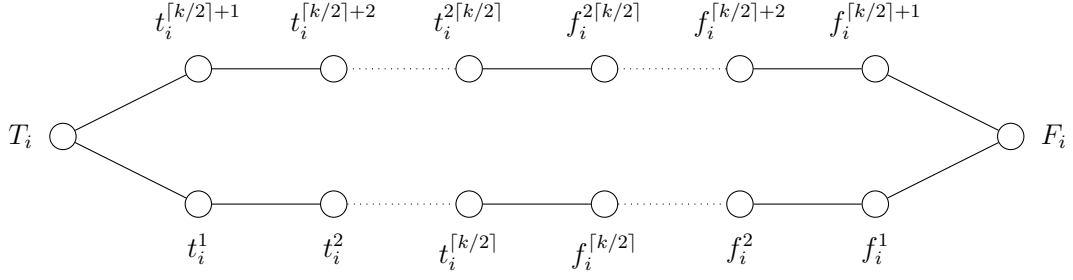
\begin{figure}[!ht]
\centering
\begin{tikzpicture}[scale=.9, transform shape]
%\tikzstyle{every node} = [draw, shape=circle]
\node [draw, shape=circle] (a1) at (0,0) {};
\node [draw, shape=circle] (a2) at (2,-1) {};
\node [draw, shape=circle] (a3) at (4,-1) {};
\node [draw, shape=circle] (a4) at (6,-1) {};
\node [draw, shape=circle] (a5) at (8,-1) {};
\node [draw, shape=circle] (a6) at (10,-1) {};
\node [draw, shape=circle] (a7) at (12,-1) {};
\node [draw, shape=circle] (a8) at (14,0) {};
\node [draw, shape=circle] (a9) at (12,1) {};
\node [draw, shape=circle] (a10) at (10,1) {};
\node [draw, shape=circle] (a11) at (8,1) {};
\node [draw, shape=circle] (a12) at (6,1) {};
\node [draw, shape=circle] (a13) at (4,1) {};
\node [draw, shape=circle] (a14) at (2,1) {};
\foreach \from/\to in {a1/a2, a2/a3, a4/a5, a6/a7, a7/a8, a8/a9, a9/a10, a11/a12, a13/a14, a14/a1}
\draw (\from) -- (\to);
\draw[dotted] (a3) -- (a4);
\draw[dotted] (a5) -- (a6);
\draw[dotted] (a10) -- (a11);
\draw[dotted] (a12) -- (a13);
\draw (-0.3,0) node[left] {$T_i$};
\draw (2,-1.3) node[below] {$t_i^1$};
\draw (4,-1.3) node[below] {$t_i^2$};
\draw (6,-1.3) node[below] {$t_i^{\lceil k/2\rceil}$};
\draw (8,-1.3) node[below] {$f_i^{\lceil k/2\rceil}$};
\draw (10,-1.3) node[below] {$f_i^2$};
\draw (12,-1.3) node[below] {$f_i^1$};
\draw (14.3,0) node[right] {$F_i$};
\draw (2,1.3) node[above] {$t_i^{\lceil k/2\rceil+1}$};
\draw (4,1.3) node[above] {$t_i^{\lceil k/2\rceil+2}$};
\draw (6,1.3) node[above] {$t_i^{2\lceil k/2\rceil}$};
\draw (8,1.3) node[above] {$f_i^{2\lceil k/2\rceil}$};
\draw (10,1.3) node[above] {$f_i^{\lceil k/2\rceil+2}$};
\draw (12,1.3) node[above] {$f_i^{\lceil k/2\rceil+1}$};
\end{tikzpicture}
\caption{The cycle $C^i$ associated to the variable $x_i$.}\label{variable}
\end{figure}
\item For every clause $Q_j\in \mathcal{Q}$, we take a star graph $S_{1,4}$ with central vertex $u_j$ and leaves $u_j^1,u_j^2,u_j^3,u_j^4$. If $k\ge 3$, then we subdivide the edge $u_ju_j^2$ until we obtain a shortest $u_j-u_j^2$ path of order $\left\lceil\frac{k}{2}\right\rceil+1$, as well as, we subdivide the edge $u_ju_j^3$ until we obtain a shortest $u_j-u_j^3$ path of order $\left\lfloor\frac{k}{2}\right\rfloor+1$ (see Figure \ref{clause}). We denote by $P(u_j^2,u_j^3)$ the shortest $u_j^2-u_j^3$ path of length $k$ obtained after subdivision.
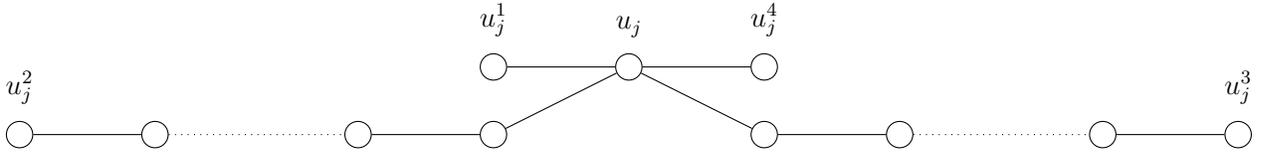
\begin{figure}[h]
\centering
\begin{tikzpicture}[scale=.9, transform shape]
%\tikzstyle{every node} = [draw, shape=circle]
\node [draw, shape=circle] (a1) at (0,0) {};
\node [draw, shape=circle] (a2) at (2,0) {};
\node [draw, shape=circle] (a3) at (5,0) {};
\node [draw, shape=circle] (a4) at (11,0) {};
\node [draw, shape=circle] (a5) at (16,0) {};
\node [draw, shape=circle] (a6) at (18,0) {};
\node [draw, shape=circle] (a7) at (7,1) {};
\node [draw, shape=circle] (a8) at (9,1) {};
\node [draw, shape=circle] (a9) at (11,1) {};
\node [draw, shape=circle] (a10) at (7,0) {};
\node [draw, shape=circle] (a11) at (13,0) {};
\foreach \from/\to in {a1/a2, a10/a8, a7/a8, a8/a9, a8/a4, a5/a6, a10/a3, a4/a11}
\draw (\from) -- (\to);
\draw[dotted] (a2) -- (a3);
\draw[dotted] (a11) -- (a5);
\draw (0,0.3) node[above] {$u_j^2$};
\draw (18,0.3) node[above] {$u_j^3$};
\draw (7,1.3) node[above] {$u_j^1$};
\draw (11,1.3) node[above] {$u_j^4$};
\draw (9,1.3) node[above] {$u_j$};
\end{tikzpicture}
\caption{The subgraph associated to the clause $Q_j$.}\label{clause}
\end{figure}

\item If a variable $x_i$ occurs as a positive literal in a clause $Q_j$, then we add the edges $T_iu_j^1$, $F_iu_j^1$ and $F_iu_j^4$ (see Figure \ref{clause-lit}).
\item If a variable $x_i$ occurs as a negative literal in a clause $Q_j$, then we add the edges $T_iu_j^1$, $F_iu_j^1$ and $T_iu_j^4$ (see Figure \ref{clause-lit}).
\begin{figure}[h]
\centering
\begin{tikzpicture}[scale=.9, transform shape]
%\tikzstyle{every node} = [draw, shape=circle]
\node [draw, shape=circle] (a1) at (0,0) {};
\node [draw, shape=circle] (a2) at (-1.5,0) {};
\node [draw, shape=circle] (a3) at (1.5,0) {};
\node [draw, shape=circle] (a4) at (-3.5,-1) {};
\node [draw, shape=circle] (a5) at (-1.5,-1) {};
\node [draw, shape=circle] (a6) at (1.5,-1) {};
\node [draw, shape=circle] (a7) at (3.5,-1) {};

\draw (0,-0.3) node[below] {$u_j$};
\draw (-1.8,-0.3) node[left] {$u_j^1$};
\draw (1.8,-0.3) node[right] {$u_j^4$};
\draw (-3.8,-1) node[left] {$u_j^2$};
\draw (3.8,-1) node[right] {$u_j^3$};

\node [draw, shape=circle] (a8) at (-9.5,2) {};
\node [draw, shape=circle] (a9) at (-8,2) {};
\node [draw, shape=circle] (a10) at (-7,2) {};
\node [draw, shape=circle] (a11) at (-6,2) {};
\node [draw, shape=circle] (a12) at (-5,2) {};
\node [draw, shape=circle] (a13) at (-3.5,2) {};
\node [draw, shape=circle] (a14) at (-5,3.5) {};
\node [draw, shape=circle] (a15) at (-6,3.5) {};
\node [draw, shape=circle] (a16) at (-7,3.5) {};
\node [draw, shape=circle] (a17) at (-8,3.5) {};

\draw (-9.5,2.3) node[above] {$T_1$};
\draw (-3.5,2.3) node[above] {$F_1$};
\draw (-8,2.2) node[above] {$t_1^1$};
\draw (-7,2.2) node[above] {$t_1^2$};
\draw (-6,2.2) node[above] {$f_1^2$};
\draw (-5,2.2) node[above] {$f_1^1$};
\draw (-8,3.7) node[above] {$t_1^3$};
\draw (-7,3.7) node[above] {$t_1^4$};
\draw (-6,3.7) node[above] {$f_1^4$};
\draw (-5,3.7) node[above] {$f_1^3$};

\node [draw, shape=circle] (a18) at (9.5,2) {};
\node [draw, shape=circle] (a19) at (8,2) {};
\node [draw, shape=circle] (a20) at (7,2) {};
\node [draw, shape=circle] (a21) at (6,2) {};
\node [draw, shape=circle] (a22) at (5,2) {};
\node [draw, shape=circle] (a23) at (3.5,2) {};
\node [draw, shape=circle] (a24) at (5,3.5) {};
\node [draw, shape=circle] (a25) at (6,3.5) {};
\node [draw, shape=circle] (a26) at (7,3.5) {};
\node [draw, shape=circle] (a27) at (8,3.5) {};

\draw (9.5,2.3) node[above] {$F_3$};
\draw (3.5,2.3) node[above] {$T_3$};
\draw (8,2.2) node[above] {$f_3^1$};
\draw (7,2.2) node[above] {$f_3^2$};
\draw (6,2.2) node[above] {$t_3^2$};
\draw (5,2.2) node[above] {$t_3^1$};
\draw (8,3.7) node[above] {$f_3^3$};
\draw (7,3.7) node[above] {$f_3^4$};
\draw (6,3.7) node[above] {$t_3^4$};
\draw (5,3.7) node[above] {$t_3^3$};

\node [draw, shape=circle] (a28) at (-3,5) {};
\node [draw, shape=circle] (a29) at (-1.5,5) {};
\node [draw, shape=circle] (a30) at (-0.5,5) {};
\node [draw, shape=circle] (a31) at (0.5,5) {};
\node [draw, shape=circle] (a32) at (1.5,5) {};
\node [draw, shape=circle] (a33) at (3,5) {};
\node [draw, shape=circle] (a34) at (1.5,6.5) {};
\node [draw, shape=circle] (a35) at (0.5,6.5) {};
\node [draw, shape=circle] (a36) at (-0.5,6.5) {};
\node [draw, shape=circle] (a37) at (-1.5,6.5) {};

\draw (-3,5.3) node[above] {$T_2$};
\draw (3,5.3) node[above] {$F_2$};
\draw (-1.5,5.2) node[above] {$t_2^1$};
\draw (-0.5,5.2) node[above] {$t_2^2$};
\draw (0.5,5.2) node[above] {$f_2^2$};
\draw (1.5,5.2) node[above] {$f_2^1$};
\draw (1.5,6.7) node[above] {$f_2^3$};
\draw (0.5,6.7) node[above] {$f_2^4$};
\draw (-0.5,6.7) node[above] {$t_2^4$};
\draw (-1.5,6.7) node[above] {$t_2^3$};

\foreach \from/\to in {a1/a2, a1/a3, a4/a5, a5/a1, a1/a6, a6/a7, a8/a9, a9/a10, a10/a11, a11/a12, a12/a13, a13/a14, a14/a15, a15/a16, a16/a17, a17/a8,
a18/a19, a19/a20, a20/a21, a21/a22, a22/a23, a23/a24, a24/a25, a25/a26, a26/a27, a27/a18, a28/a29, a29/a30, a30/a31, a31/a32, a32/a33, a33/a34, a34/a35, a35/a36, a36/a37, a37/a28, a2/a8, a2/a13, a3/a13, a2/a28, a2/a33, a3/a28, a2/a18, a2/a23, a3/a23}
\draw (\from) -- (\to);
\end{tikzpicture}
\caption{The subgraph associated to the clause $Q_j=(x_1\vee \overline{x_2} \vee\overline{x_3})$ (taking $k=4$).}\label{clause-lit}
\end{figure}
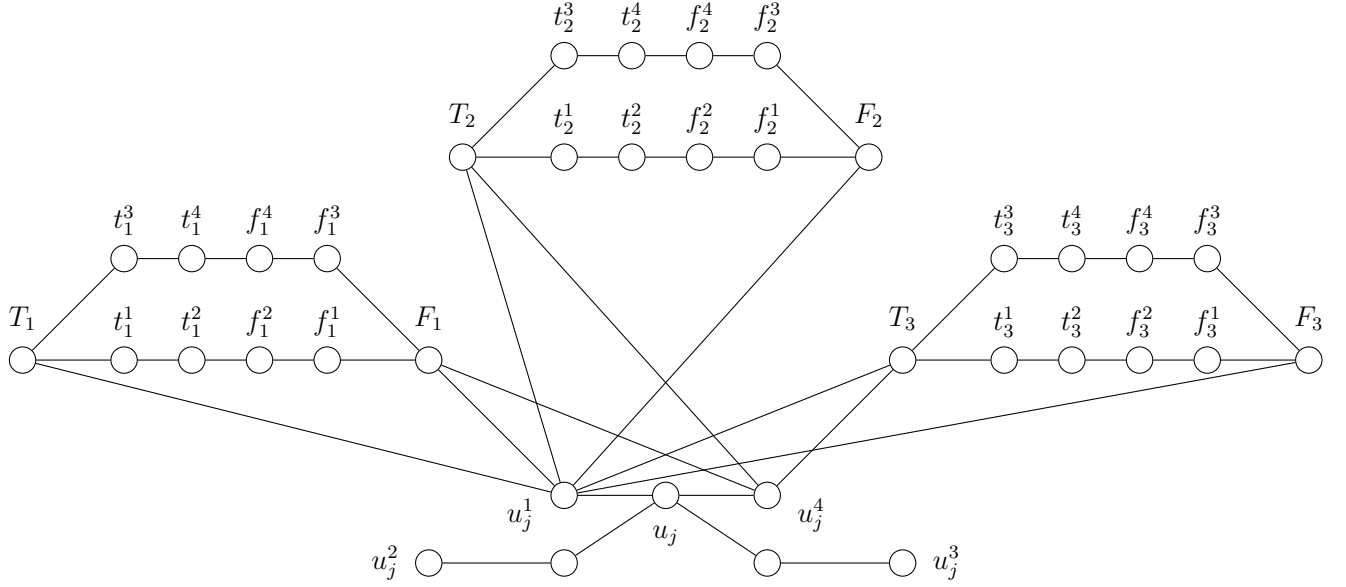
\item Finally, for every $l\in \{1,...,n\}$ such that $x_l$ and $\overline{x_l}$ do not occur in a clause $Q_j$ we add the the edges $T_lu_j^1$, $T_lu_j^4$, $F_lu_j^1$ and $F_lu_j^4$.
\end{itemize}
Notice that the graph $G_F$ obtained from the procedure above has order $n\left(4\left\lceil\frac{k}{2}\right\rceil+2\right)+s(k+3)$. Also, it is straightforward to observe that given the formula $\mathcal{F}$, the graph $G_F$ can be constructed in polynomial time. Next we prove that $\mathcal{F}$ is satisfiable if and only if $\dim_k(G_F)=k(n+s)$. To do so, we first notice some properties of $G_F$.

\begin{remark}\label{rem1}
Let $x_i\in X$. Then there exist two different vertices $a,b\in V(C^i)$ such that they are distinguished only by vertices of the cycle $C^i$ and, as consequence, for any $k$-metric basis $S$ of $G_F$, we have that $|S\cap V(C^i)|\ge k$.
\end{remark}

\begin{proof}
To observe that it is only necessary to take the two vertices of $C^i$ adjacent to $T_i$ or adjacent to $F_i$.
\end{proof}

\begin{remark}\label{rem2}
Let $Q_j\in \mathcal{Q}$. Then there exist two different vertices $x,y$ in the shortest $u_j^2-u_j^3$ path such that they are distinguished only by vertices of the itself shortest $u_j^2-u_j^3$ path and, as consequence, for any $k$-metric basis $S$ of $G_F$, we have that $\vert S\cap V(P(u_j^2,u_j^3))\vert \ge k$.
\end{remark}

\begin{proof}
To observe that it is only necessary to take the two vertices of $P(u_j^2,u_j^3)$  adjacent to $u_j$.
\end{proof}

\begin{proposition}\label{eq-1}
Let $\mathcal{F}$ be an arbitrary input to $3$-SAT problem. Then the graph $G_F$ associated to $\mathcal{F}$ satisfies that $\dim_k(G_F)\ge k(n+s)$.
\end{proposition}

\begin{proof}
As a consequence of Remarks \ref{rem1} and \ref{rem2} we obtain that for every variable $x_i\in X$ and for every clause $Q_j\in \mathcal{Q}$ the set of vertices of $G_F$ associated to each variable or clause, contains at least $k$ vertices of every $k$-metric basis for $G_F$. Thus, the result follows.
\end{proof}

\begin{theorem}
$k$-METRIC DIMENSION PROBLEM is $NP$-complete.
\end{theorem}

\begin{proof}
Let $\mathcal{F}$ be an arbitrary input to $3$-SAT problem having more than three variables and let $G_F$ be the graph associated to $\mathcal{F}$. We shall show that $\mathcal{F}$ is satisfiable if and only if $\dim_k(G_F)=k(n+s)$.

We first assume that $\mathcal{F}$ is satisfiable. %We shall prove that $\dim_k(G_F)=k(n+s)$.
From Proposition \ref{eq-1} we have that $\dim_k(G_F)\ge k(n+s)$. Now, based on a satisfying assignment of $\mathcal{F}$, we shall give a set $S$ of vertices of $G_F$, of cardinality $\vert S\vert=k(n+s)$,  which is $k$-metric generator.

Suppose we have a satisfying assignment for $\mathcal{F}$. For every clause $Q_j\in \mathcal{Q}$ we add to $S$ all the vertices of the set $V(P(u_j^2,u_j^3))-\{u_j\}$. For a variable $x_i\in X$ we consider the following. If the value of $x_i$ is true, then we add to $S$ the vertices $t_i^1,t_i^2,...,t_i^{2\left\lceil\frac{k}{2}\right\rceil}$. On the contrary, if the value of $x_i$ is false, then we add to $S$ the vertices $f_i^1,f_i^2,...,f_i^{2\left\lceil\frac{k}{2}\right\rceil}$.

We shall show that $S$ is a $k$-metric generator for $G_F$. Let $a,b$ be two different vertices of $G_F$. We consider the following cases.\\

%Case 1. $a,b\in V(C^i)$ for some $i\in \{1,...,n\}$. Hence, there exists at most one vertex $y\in S\cap V(C^i)$ such that $d(a,y)=d(b,y)$. Since %$|V(C^i)|=4\left\lceil\frac{k}{2}\right\rceil+2$, we have that $d(a,T_i)\ne d(b,T_i)$ and $d(a,F_i)\ne d(b,F_i)$. Thus, for every $w_1\in S-V(C^i)$ it follows %that $d(a,w_1)\ne d(b,w_1)$ and $a,b$ are resolved by more than $k$ vertices of $S$.

\noindent Case 1. $a,b\in V(C^i)$ for some $i\in \{1,...,n\}$. Hence, there exists at most one vertex $y\in S\cap V(C^i)$ such that $d(a,y)=d(b,y)$. If $d(a,w)\ne d(b,w)$ for every vertex $w\in S\cap V(C^i)$, then since $|S\cap V(C^i)|=k$, we have that $|\mathcal{D}_G(a,b)\cap S|=k$. On the other hand, if there exist one vertex $y\in S\cap V(C^i)$ such that $d(a,y)=d(b,y)$, then $d(a,T_i)\ne d(b,T_i)$ and $d(a,F_i)\ne d(b,F_i)$. Thus, for every $w\in S-V(C^i)$ it follows that $d(a,w)\ne d(b,w)$ and $a,b$ are distinguished by more than $k$ vertices of $S$.\\

\noindent Case 2. $a,b\in V(P(u_j^2,u_j^3))$. Hence, there exists at most one vertex $y'\in S\cap V(P(u_j^2,u_j^3))$ such that $d(a,y')=d(b,y')$. But, in this case, $d(a,u_j)\ne d(b,u_j)$ and so  for every $w\in S-V(P(u_j^2,u_j^3))$ it follows that $d(a,w)\ne d(b,w)$ and $a,b$ are distinguished by at least $k$ vertices of $S$.\\

\noindent Case 3. $a=u_j^1$ and $b=u_j^4$. Since the clause $Q_j$ is satisfied, there exists $i\in \{1,...,n\}$, \emph{i.e}, a variable $x_i$ occurring in the clause $Q_j$ such that either
\begin{itemize}
\item $a\sim T_i$, $b\not\sim T_i$ and $S\cap V(C^i)=\left\{t_i^1,t_i^2,...,t_i^{2\left\lceil\frac{k}{2}\right\rceil}\right\}$, \emph{i.e}, a variable $x_i$ occurring as a positive literal in $Q_j$ and has the value true in the assignment, or
\item $a\sim F_i$, $b\not\sim F_i$ and $S\cap V(C^i)=\left\{f_i^1,f_i^2,...,f_i^{2\left\lceil\frac{k}{2}\right\rceil}\right\}$, \emph{i.e}, a variable $x_i$ occurring as a negative literal in $Q_j$ and has the value false in the assignment.
\end{itemize}
Thus, in any case we have that for every $w\in S\cap V(C^i)$ it follows $d(a,w)<d(b,w)$ and $a,b$ are distinguished by at least $k$ vertices of $S$.\\

\noindent Case 4. $a\in V(C^i)$ and $b\in V(C^l)$ for some $i,l\in \{1,...,n\}$, $i\ne l$. In this case, if there is a vertex $z\in S\cap V(C^i)$ such that  $d(a,z)=d(b,z)$, then for every vertex $w\in S\cap V(C^l)$ it follows that $d(a,w)\ne d(b,w)$. So $a,b$ are resolved by at least $k$ vertices of $S$.\\

\noindent Case 5. $a\in V(C^i)$ and $b\in V(P(u_j^2,u_j^3))$. It is similar to the case above.\\

\noindent Case 6. $a\in \{u_j^1,u_j^4\}$ and $b\notin \{u_j^1,u_j^4\}$.
%Since every shortest $w_6-z_6$ path, where $w_6\in V(C^l)$ and $z_6\in V(P(u_j^2,u_j^3))$ for some $l\in \{1,\ldots,n\}$, pass throughout $u_j^1$ or $u_j^4$, we have that $a,b$ are resolved either by $k$ vertices of $S\cap P(u_i^2,u_i^3)$ or by $k$ vertices of $S\cap V(C^l)$). That is,
 If $b\in V(C^i)$, for some $i\in \{1,\ldots,n\}$, then all elements of $S\cap V(P(u_j^2,u_j^3)$ distinguish $a,b$. Now, let $w$ be one of the two vertices adjacent to $u_j$ in $P(u_j^2,u_j^3)$. If $b\in V(P(u_j^2,u_j^3))-\{w\}$, then all elements of $S\cap V(C^i)$ distinguish $a,b$. On the other hand, since $n\ge 4$, if $b=w$, then  there exists a variable $x_l$ not occurring in the clause $Q_j$. Thus, the vertex $a$ is adjacent to $T_l$ and to $F_l$ and, as a consequence, the vertices of $S\cap V(C^l)$ distinguish $a,b$.

As a consequence of the cases above, we have that $S$ is a $k$-metric generator for $G_F$. Therefore, $\dim_k(G_F)=k(n+s)$.

%Let $S$ be a $k$-metric basis of $G_F$. Hence, for every variable $x_i\in X$ and for every clause $Q_j\in \mathcal{Q}$, the set of vertices of $G_F$ associated %to each variable or clause, contains exactly $k$ vertices of $S$. Now, according to the construction of the graph $G_F$, we have that for every clause $Q_i\in %\mathcal{Q}$ and for every variable $x_j\in X$ occurring in the clause $Q_i$ we have that ${\cal D}_G^*(u_i^1,u_i^4)\cap %V(C^j)=\{t_j^1,t_j^2,...,t_j^{2\left\lceil\frac{k}{2}\right\rceil}\}$ or ${\cal D}_G^*(u_i^1,u_i^4)\cap %V(C^j)=\{f_j^1,f_j^2,...,f_j^{2\left\lceil\frac{k}{2}\right\rceil}\}$. Thus, for every clause $Q_i\in \mathcal{Q}$ there exists a variable $x_j\in X$ occurring %in the clause $Q_i$ such that either $S\cap V(C_i)\cap \{t_j^1,t_j^2,...,t_j^{2\left\lceil\frac{k}{2}\right\rceil}\}=\emptyset$ or $S\cap %\{f_j^1,f_j^2,...,f_j^{2\left\lceil\frac{k}{2}\right\rceil}\}=\emptyset$.

Next we prove that, if $\dim_k(G_F)=k(n+s)$, then $\mathcal{F}$ is satisfiable. To this end, we show that there exists a $k$-metric basis $S$ of $G_F$ such that we can set an assignment of the variables, so that $\mathcal{F}$ is satisfiable. We take $S$ in the same way as the $k$-metric generator for $G_F$ described above. Since $S$ is a $k$-metric generator for $G_F$ of cardinality $k(n+s)$, it is also a $k$-metric basis. Note that for any cycle $C_i$ either $S\cap V(C_i)=\left\{t_i^1,t_i^2,...,t_i^{2\left\lceil\frac{k}{2}\right\rceil}\right\}$ or $S\cap V(C_i)=\left\{f_i^1,f_i^2,...,f_i^{2\left\lceil\frac{k}{2}\right\rceil}\right\}$.

In this sense, we set an assignment of the variables as follows. Given a variable $x_i\in X$, if $S\cap \left\{t_i^1,t_i^2,...,t_i^{2\left\lceil\frac{k}{2}\right\rceil}\right\}=\emptyset$, then we set $x_i$ to be false. Otherwise we set $x_i$ to be true.
We claim that this assignment satisfies $\mathcal{F}$.

Consider any clause $Q_j\in \mathcal{Q}$ and let $x_{j_1},x_{j_2},x_{j_3}$ the variables occurring in $Q_j$. Recall that for each clause $Q_h$, we have that $S\cap V(P(u_h^2,u_h^3))=V(P(u_h^2,u_h^3))-\{u_h\}$. Besides any vertex of $V(C_l)$ associated to a variable $x_l$, $l\ne j_1,j_2,j_3$, nor any vertex of $S\cap V(P(u_h^2,u_h^3))$ associated to a clause $Q_h$, distinguishes the vertices $u_j^1$ and $u_j^4$. Thus $u_j^1$ and $u_j^4$ must be distinguished by at least $k$ vertices belonging to $V(C_{j_1})\cup V(C_{j_2})\cup V(C_{j_3})$ associated to the variables $x_{j_1},x_{j_2},x_{j_3}$.

Now, according to the way in which we have added the edges between the vertices $T_{j_1}$, $T_{j_2}$, $T_{j_3}$, $F_{j_1}$, $F_{j_2}$, $F_{j_3}$ and $u_j^1,u_j^4$, we have that $u_j^1$ and $u_j^4$ are distinguished by at least $k$ vertices of $S$ if and only if one of the following statements holds.
\begin{itemize}
\item There exists $l\in \{1,2,3\}$ for which the variable $x_{j_l}$ occurs as a negative literal in the clause $Q_j$ and $S\cap \left\{t_j^1,t_j^2,...,t_j^{2\left\lceil\frac{k}{2}\right\rceil}\right\}=\emptyset$ (in such a case $x_{j_l}$ is set to be false).
\item There exists $l\in \{1,2,3\}$ for which the variable $x_{j_l}$ occurs as a positive literal in the clause $Q_j$ and $S\cap \left\{t_j^1,t_j^2,...,t_j^{2\left\lceil\frac{k}{2}\right\rceil}\right\}\ne \emptyset$ (in such a case $x_{j_l}$ is set to be true).
\end{itemize}
As a consequence of two cases above, we have that if at least $k$ vertices of $S$ distinguish $u_j^1,u_j^4$, then the setting of $x_{j_l}$, $l\in \{1,2,3\}$, is such that it satisfies the clause $Q_j$. Therefore $\mathcal{F}$ is satisfiable.
\end{proof}

As a consequence of the theorem above we have the following result.

\begin{corollary}
The problem of finding the $k$-metric dimension of graphs is $NP$-hard.
\end{corollary}

\section{The particular case of trees}\label{k-metric-dimension-tree}

In order to continue presenting our results, we need to introduce some definitions. A vertex of degree at least three in a tree $T$ is called a \emph{major vertex} of $T$. Any leaf $u$ of $T$ is said to be a \emph{terminal vertex} of a major vertex $v$ of $T$, if $d_{T}(u,v)<d_{T}(u,w)$ for every other major vertex $w$ of $T$. The \emph{terminal degree} $\ter(v)$ of a major vertex $v$ is the number of terminal vertices of $v$. A major vertex $v$ of $T$ is an \emph{exterior major vertex} of $T$ if it has positive terminal degree. Let $\mathcal{M}(T)$ be the set of exterior major vertices of $T$ having terminal degree greater than one.

Given $w\in \mathcal{M}(T)$ and a terminal vertex $u_{j}$ of $w$, we denote by  $P(u_j,w)$ the shortest path that starts at $u_{j}$ and ends at $w$. Let $l(u_{j},w)$ be the length of $P(u_j,w)$. Now, given $w\in \mathcal{M}(T)$ and two terminal vertices $u_{j},u_{r}$ of $w$ we denote by  $P(u_{j},w,u_{r})$  the shortest path from $u_{j}$ to $u_{r}$ containing  $w$,
 and by  $\varsigma(u_{j},u_{r})$  the length of $P(u_{j},w,u_{r})$. Notice that, by the definition of exterior major vertex, $P(u_{j},w,u_{r})$ is obtained by concatenating the paths $P(u_{j},w)$ and $P(u_{r},w)$, where $w$ is the only vertex of degree greater than two lying on these paths.

Finally, given $w\in \mathcal{M}(T)$ and the set of terminal vertices $U=\{u_{1},u_{2},\ldots,u_{k}\}$ of $w$, for $j\not=r$ we define $$\varsigma(w)=\displaystyle\min_{u_{j},u_{r}\in U}\{\varsigma(u_{j},u_{r})\}$$ and  $$l(w)=\displaystyle\min_{u_{j}\in U}\{l(u_{j},w)\}.$$ From the local parameters above we define the following global parameter $$\varsigma(T)=\min_{w\in \mathcal{M}(T)}\{\varsigma(w)\}.$$

An example of a tree $T$ which helps to better understand the notation above is given in Figure~\ref{figTreeNotation}. In such a case we have that $\mathcal{M}(T)=\{6,12,26\}$,  $\{1,4\}$ is the set of terminal vertices of $6$, $\{9,11\}$ is the set of terminal vertices of $12$ and $\{15,20,23\}$ is the set of terminal vertices of $26$. For instance, for the vertex $26$ we have that $l(26)=\min\{l(15,26),l(20,26),l(23,26)\}=\min\{5,3,3\}=3$ and $\varsigma(26)=\min\{\varsigma(15,20),\varsigma(15,23),\varsigma(20,23)\}=\min\{8,8,6\}=6$. Analogously, we deduce that $l(6)=2$, $\varsigma(6)=5$, $l(12)=1$ and $\varsigma(12)=3$. Therefore, we conclude that $\varsigma(T)=\min\{\varsigma(6),\varsigma(12),\varsigma(26)\}=\min\{5,3,6\}=3$.

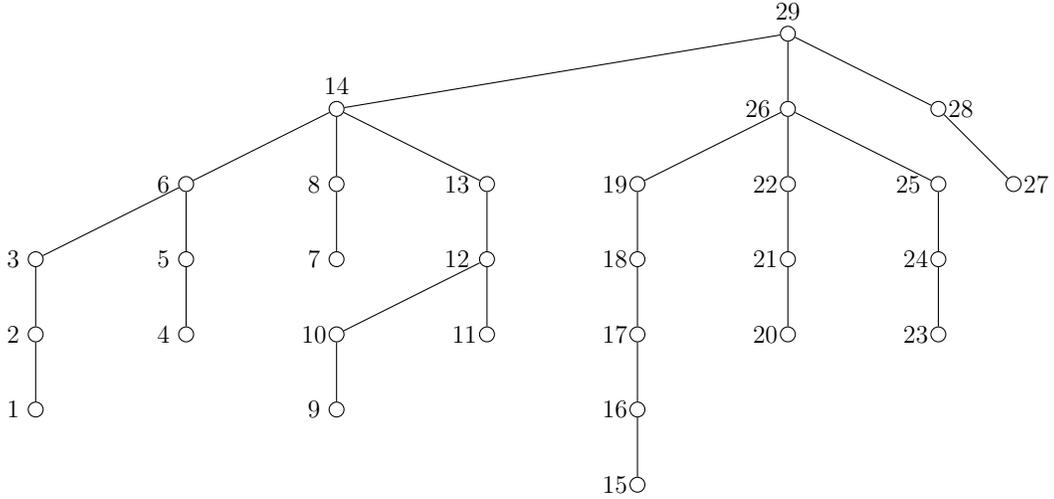
\begin{figure}[!ht]
\centering
\begin{tikzpicture}[transform shape, inner sep = .7mm]
\node [draw=black, shape=circle, fill=white] (v0) at (0,0) {};
\node [draw=black, shape=circle, fill=white] (v01) at (-6,-1) {};
\node [draw=black, shape=circle, fill=white] (v02) at (0,-1) {};
\node [draw=black, shape=circle, fill=white] (v03) at (2,-1) {};
\node [draw=black, shape=circle, fill=white] (v11) at (-8,-2) {};
\node [draw=black, shape=circle, fill=white] (v12) at (-4,-2) {};
\node [draw=black, shape=circle, fill=white] (v13) at (-6,-2) {};
\node [draw=black, shape=circle, fill=white] (v111) at (-10,-3) {};
\node [draw=black, shape=circle, fill=white] (v112) at (-8,-3) {};
\node [draw=black, shape=circle, fill=white] (v1111) at (-10,-4) {};
\node [draw=black, shape=circle, fill=white] (v11111) at (-10,-5) {};
\node [draw=black, shape=circle, fill=white] (v1121) at (-8,-4) {};
\node [draw=black, shape=circle, fill=white] (v121) at (-4,-3) {};
\node [draw=black, shape=circle, fill=white] (v1211) at (-6,-4) {};
\node [draw=black, shape=circle, fill=white] (v1212) at (-4,-4) {};
\node [draw=black, shape=circle, fill=white] (v12111) at (-6,-5) {};
\node [draw=black, shape=circle, fill=white] (v131) at (-6,-3) {};
\node [draw=black, shape=circle, fill=white] (v21) at (-2,-2) {};
\node [draw=black, shape=circle, fill=white] (v22) at (0,-2) {};
\node [draw=black, shape=circle, fill=white] (v23) at (2,-2) {};
\node [draw=black, shape=circle, fill=white] (v211) at (-2,-3) {};
\node [draw=black, shape=circle, fill=white] (v2111) at (-2,-4) {};
\node [draw=black, shape=circle, fill=white] (v21111) at (-2,-5) {};
\node [draw=black, shape=circle, fill=white] (v211111) at (-2,-6) {};
\node [draw=black, shape=circle, fill=white] (v221) at (0,-3) {};
\node [draw=black, shape=circle, fill=white] (v2211) at (0,-4) {};
\node [draw=black, shape=circle, fill=white] (v231) at (2,-3) {};
\node [draw=black, shape=circle, fill=white] (v2311) at (2,-4) {};
\node [draw=black, shape=circle, fill=white] (v31) at (3,-2) {};

\draw (v0) -- (v01) -- (v11) -- (v111) -- (v1111) -- (v11111);
\draw (v11) -- (v112) -- (v1121);
\draw (v01) -- (v12) -- (v121) -- (v1211) -- (v12111);
\draw (v01) -- (v13) -- (v131);
\draw (v121) -- (v1212);
\draw (v0) -- (v02) -- (v21) -- (v211) -- (v2111) -- (v21111) -- (v211111);
\draw (v02) -- (v22) -- (v221) -- (v2211);
\draw (v02) -- (v23) -- (v231) -- (v2311);
\draw (v0) -- (v03) -- (v31);

\node [scale=.8] at ([xshift=-.3 cm]v11111) {$1$};
\node [scale=.8] at ([xshift=-.3 cm]v1111) {$2$};
\node [scale=.8] at ([xshift=-.3 cm]v111) {$3$};
\node [scale=.8] at ([xshift=-.3 cm]v1121) {$4$};
\node [scale=.8] at ([xshift=-.3 cm]v112) {$5$};
\node [scale=.8] at ([xshift=-.3 cm]v11) {$6$};
\node [scale=.8] at ([xshift=-.3 cm]v131) {$7$};
\node [scale=.8] at ([xshift=-.3 cm]v13) {$8$};
\node [scale=.8] at ([xshift=-.3 cm]v12111) {$9$};
\node [scale=.8] at ([xshift=-.3 cm]v1211) {$10$};
\node [scale=.8] at ([xshift=-.3 cm]v1212) {$11$};
\node [scale=.8] at ([xshift=-.4 cm]v121) {$12$};
\node [scale=.8] at ([xshift=-.4 cm]v12) {$13$};
\node [scale=.8] at ([yshift=.3 cm]v01) {$14$};
\node [scale=.8] at ([xshift=-.3 cm]v211111) {$15$};
\node [scale=.8] at ([xshift=-.3 cm]v21111) {$16$};
\node [scale=.8] at ([xshift=-.3 cm]v2111) {$17$};
\node [scale=.8] at ([xshift=-.3 cm]v211) {$18$};
\node [scale=.8] at ([xshift=-.3 cm]v21) {$19$};
\node [scale=.8] at ([xshift=-.3 cm]v2211) {$20$};
\node [scale=.8] at ([xshift=-.3 cm]v221) {$21$};
\node [scale=.8] at ([xshift=-.3 cm]v22) {$22$};
\node [scale=.8] at ([xshift=-.3 cm]v2311) {$23$};
\node [scale=.8] at ([xshift=-.3 cm]v231) {$24$};
\node [scale=.8] at ([xshift=-.4 cm]v23) {$25$};
\node [scale=.8] at ([xshift=-.4 cm]v02) {$26$};
\node [scale=.8] at ([xshift=.3 cm]v31) {$27$};
\node [scale=.8] at ([xshift=.3 cm]v03) {$28$};
\node [scale=.8] at ([yshift=.3 cm]v0) {$29$};
\end{tikzpicture}
\caption{A tree $T$ where $\varsigma(T)=3$. Note that vertices are labeled through a post-order traversal.}
\label{figTreeNotation}
\end{figure}

%%%%%%%%%%%%%%%%%%

\subsection{On $k$-metric dimensional trees different from paths}

In this section we focus on finding the positive integer $k$ for which a tree is $k$-metric dimensional. We now show a result presented in \cite{Estrada-Moreno2013} that allow us to consider only those trees that are not paths.

\begin{theorem}{\em \cite{Estrada-Moreno2013}}
$\ $A graph $G$ of order $n\ge 3$ is $(n-1)$-metric dimensional if and only if $G$ is a path or $G$ is an odd cycle.
\end{theorem}

The following theorem presented in \cite{Estrada-Moreno2013} is the base of the algorithm presented in this subsection.

\begin{theorem}{\em \cite{Estrada-Moreno2013}}\label{theoTreeDimK}
If $T$ is a $k$-metric dimensional tree different from a path, then $k=\varsigma(T)$.
\end{theorem}

Now we consider the problem of finding the integer $k$ such that a tree $T$ of order $n$ is $k$-metric dimensional.
$$\begin{tabular}{|l|}
  \hline
  \mbox{$k$-DIMENSIONAL TREE PROBLEM}\\
  \mbox{INSTANCE: A tree $T$ different from a path of order $n$}\\
  \mbox{PROBLEM: Find the integer $k$, $2\le k\le n-1$, such that $T$ is $k$-metric dimensional}\\
  \hline
\end{tabular}$$

\noindent\textbf{Algorithm 1: }\\
{\bf Input:} A tree $T$ different from a path rooted in a major vertex $v$.\\
{\bf Output:} The value $k$ for which $T$ is $k$-metric dimensional.
\begin{enumerate}
\item For any vertex $u\in V(T)$ visited by post-order traversal as shown in Figure \ref{figTreeNotation}, assign a pair $(a_u,b_u)$ in the following way:
\begin{enumerate}
\item If $u$ does not have any child ($u$ is a leaf), then $a_u=1$ and $b_u=\infty$.
\item If $u$ has only one child ($u$ has degree $2$), then $a_u=a_{u'}+1$ and $b=b_{u'}$, where the pair $(a_{u'},b_{u'})$ was assigned to the child vertex of $u$. Note that $a_{u'}$ can be $\infty$. Thus, in such case, $a_u=\infty$.
\item If $u$ has at least two children ($u$ is a major vertex), then $a_u=\infty$ and $b_u=\min\{a_{u_1}+a_{u_2}, b_{min}\}$, where $a_{u_1}$ and $a_{u_2}$ are the two minimum values among all possible pairs $(a_{u_i},b_{u_i})$ assigned to the children of $u$, and $b_{min}$ is the minimum value among all the $b_{u_i}$'s.
\end{enumerate}
\item The value $k$ for which $T$ is $k$-metric dimensional equals $b_v$ (the second element of the pair assigned to the root $v$).
\end{enumerate}

Figure \ref{figTreeDimensional} shows an example of a run of Algorithm 1 for the tree shown in Figure \ref{figTreeNotation}.

\begin{figure}[!ht]
\centering
\begin{tikzpicture}[transform shape, inner sep = .7mm]
\node [draw=black, shape=circle, fill=white] (v0) at (0,0) {};
\node [draw=black, shape=circle, fill=white] (v01) at (-6,-1) {};
\node [draw=black, shape=circle, fill=white] (v02) at (0,-1) {};
\node [draw=black, shape=circle, fill=white] (v03) at (2,-1) {};
\node [draw=black, shape=circle, fill=white] (v11) at (-8,-2) {};
\node [draw=black, shape=circle, fill=white] (v12) at (-4,-2) {};
\node [draw=black, shape=circle, fill=white] (v13) at (-6,-2) {};
\node [draw=black, shape=circle, fill=white] (v111) at (-10,-3) {};
\node [draw=black, shape=circle, fill=white] (v112) at (-8,-3) {};
\node [draw=black, shape=circle, fill=white] (v1111) at (-10,-4) {};
\node [draw=black, shape=circle, fill=white] (v11111) at (-10,-5) {};
\node [draw=black, shape=circle, fill=white] (v1121) at (-8,-4) {};
\node [draw=black, shape=circle, fill=white] (v121) at (-4,-3) {};
\node [draw=black, shape=circle, fill=white] (v1211) at (-6,-4) {};
\node [draw=black, shape=circle, fill=white] (v1212) at (-4,-4) {};
\node [draw=black, shape=circle, fill=white] (v12111) at (-6,-5) {};
\node [draw=black, shape=circle, fill=white] (v131) at (-6,-3) {};
\node [draw=black, shape=circle, fill=white] (v21) at (-2,-2) {};
\node [draw=black, shape=circle, fill=white] (v22) at (0,-2) {};
\node [draw=black, shape=circle, fill=white] (v23) at (2,-2) {};
\node [draw=black, shape=circle, fill=white] (v211) at (-2,-3) {};
\node [draw=black, shape=circle, fill=white] (v2111) at (-2,-4) {};
\node [draw=black, shape=circle, fill=white] (v21111) at (-2,-5) {};
\node [draw=black, shape=circle, fill=white] (v211111) at (-2,-6) {};
\node [draw=black, shape=circle, fill=white] (v221) at (0,-3) {};
\node [draw=black, shape=circle, fill=white] (v2211) at (0,-4) {};
\node [draw=black, shape=circle, fill=white] (v231) at (2,-3) {};
\node [draw=black, shape=circle, fill=white] (v2311) at (2,-4) {};
\node [draw=black, shape=circle, fill=white] (v31) at (3,-2) {};

\draw (v0) -- (v01) -- (v11) -- (v111) -- (v1111) -- (v11111);
\draw (v11) -- (v112) -- (v1121);
\draw (v01) -- (v12) -- (v121) -- (v1211) -- (v12111);
\draw (v01) -- (v13) -- (v131);
\draw (v121) -- (v1212);
\draw (v0) -- (v02) -- (v21) -- (v211) -- (v2111) -- (v21111) -- (v211111);
\draw (v02) -- (v22) -- (v221) -- (v2211);
\draw (v02) -- (v23) -- (v231) -- (v2311);
\draw (v0) -- (v03) -- (v31);

\node [scale=.8] at ([xshift=-.6 cm]v11111) {$(1,\infty)$};
\node [scale=.8] at ([xshift=-.6 cm]v1111) {$(2,\infty)$};
\node [scale=.8] at ([xshift=-.6 cm]v111) {$(3,\infty)$};
\node [scale=.8] at ([xshift=-.6 cm]v1121) {$(1,\infty)$};
\node [scale=.8] at ([xshift=-.6 cm]v112) {$(2,\infty)$};
\node [scale=.8] at ([xshift=-.6 cm]v11) {$(\infty,5)$};
\node [scale=.8] at ([xshift=-.6 cm]v131) {$(1,\infty)$};
\node [scale=.8] at ([xshift=-.6 cm]v13) {$(2,\infty)$};
\node [scale=.8] at ([xshift=-.6 cm]v12111) {$(1,\infty)$};
\node [scale=.8] at ([xshift=-.6 cm]v1211) {$(2,\infty)$};
\node [scale=.8] at ([xshift=-.6 cm]v1212) {$(1,\infty)$};
\node [scale=.8] at ([xshift=-.7 cm]v121) {$(\infty,3)$};
\node [scale=.8] at ([xshift=-.7 cm]v12) {$(\infty,3)$};
\node [scale=.8] at ([yshift=.3 cm]v01) {$(\infty,3)$};
\node [scale=.8] at ([xshift=-.6 cm]v211111) {$(1,\infty)$};
\node [scale=.8] at ([xshift=-.6 cm]v21111) {$(2,\infty)$};
\node [scale=.8] at ([xshift=-.6 cm]v2111) {$(3,\infty)$};
\node [scale=.8] at ([xshift=-.6 cm]v211) {$(4,\infty)$};
\node [scale=.8] at ([xshift=-.6 cm]v21) {$(5,\infty)$};
\node [scale=.8] at ([xshift=-.6 cm]v2211) {$(1,\infty)$};
\node [scale=.8] at ([xshift=-.6 cm]v221) {$(2,\infty)$};
\node [scale=.8] at ([xshift=-.6 cm]v22) {$(3,\infty)$};
\node [scale=.8] at ([xshift=-.6 cm]v2311) {$(1,\infty)$};
\node [scale=.8] at ([xshift=-.6 cm]v231) {$(2,\infty)$};
\node [scale=.8] at ([xshift=-.6 cm]v23) {$(3,\infty)$};
\node [scale=.8] at ([xshift=-.6 cm]v02) {$(\infty,6)$};
\node [scale=.8] at ([xshift=.6 cm]v31) {$(1,\infty)$};
\node [scale=.8] at ([xshift=.6 cm]v03) {$(2,\infty)$};
\node [scale=.8] at ([yshift=.3 cm]v0) {$(\infty,3)$};
\end{tikzpicture}
\caption{Algorithm 1 yields that this tree is $3$-metric dimensional.}
\label{figTreeDimensional}
\end{figure}

\begin{remark}
Let $T$ be a tree different from a path of order $n$. Algorithm 1 computes the integer $k$, $2\le k\le n-1$, such that $T$ is $k$-metric dimensional.
\end{remark}

\begin{proof}
Let $v$ be the major vertex taken as the root of the tree $T$ different from a path, and let  $(a_v,b_v)$ be the pair stored in $v$ by Algorithm 1. We show that $b_v=\varsigma(T)$. Since $v$ is a major vertex, it has at least three children. Let $t\ge 3$ be the number of children of $v$ and let $S_1,\ldots,S_t$ be the subtrees whose roots are the children $v_1,\ldots,v_t$ of $v$, respectively. We differentiate two cases:
\begin{enumerate}
\item There exist at least two subtrees that are paths. In this case $v\in\mathcal{M}(T)$. Let $S_1,\ldots,S_{t'}$ be the subtrees that are paths, where $2\le t'\le t$. In this case, after running Algorithm 1, each root $v_i$ of $S_i$, $1\le i\le t'$, stores the pair $(a_{v_i},\infty)$, where $a_{v_i}$ is the number of vertices of $S_i$. Note that $\varsigma(v)=a_{v_1}+a_{v_2}$, where $a_{v_1}$ and $a_{v_2}$ are the two minimum values among all $a_{v_i}$'s belonging to the pairs $(a_{v_i},b_{v_i})$ stored by the children of $v$ such that $1\le i\le t'$. If $t'=t$, then $v$ is the only exterior major vertex of $T$, and Algorithm 1 stores in $v$ the pair $(a_v,b_v)=(\infty,\varsigma(v))=(\infty,\varsigma(T))$. Assume now that $t'<t$. Thus, there exists at least one subtree that is not path. Let $S_{t'+1},\ldots,S_t$ be the subtrees that are not paths. For each root $v_i$ of $S_i$, $t'+1\le i\le t$,  if $v_i$ is a major vertex, then we take the vertex $v_i'=v_i$. Otherwise, $v_i'$ is the first descendant of $v_i$ that is a major vertex. In this case, Algorithm 1 recursively stores in $v_i'$ the pair $(\infty, b_{v_i'})$, where $\displaystyle b_{v_i'}=\min_{v'\in\mathcal{M}(T)\cap V(S_i)}\{\varsigma(v')\}$. In both cases, $b_{v_i}=b_{v_i'}$, where $(\infty, b_{v_i})$ is the pair stored in $v_i$ by Algorithm 1. Therefore, by Algorithm 1, the root $v$ stores the pair $(a_v,b_v)=(\infty, \min\{\varsigma(v), b_{min}\})=(\infty,\varsigma(T))$, where $\displaystyle b_{min}=\min_{t'+1\le i\le t}\{b_i\}$.
\item There exists at most one subtree that is a path. In this case $v\notin\mathcal{M}(T)$. Let $S_1,\ldots,S_t'$ be the subtrees that are not paths, where $1\le t'\le t$. For each root $v_i$ of $S_i$, $1\le i\le t'$, if $v_i$ is a major vertex, then we take the vertex $v_i'=v_i$. Otherwise, $v_i'$ is the first descendant of $v_i$ that is a major vertex. In this case, Algorithm 1 recursively stores in $v_i'$ the pair $(\infty, b_{v_i'})$, where $\displaystyle b_{v_i'}=\min_{v'\in\mathcal{M}(T)\cap V(S_i)}\{\varsigma(v')\}$. In both cases, $b_{v_i}=b_{v_i'}$, where $(\infty, b_{v_i})$ is the pair stored in $v_i$ by Algorithm 1. Note in this case, at least one of two minimum values among all $a_{v_i}$ of pairs $(a_{v_i},b_{v_i})$ stored by the children of $v$ is infinity. Therefore, by Algorithm 1, $v$ stores the pair $(a_v,b_v)=(\infty, b_{min})=(\infty,\varsigma(T))$, where $\displaystyle b_{min}=\min_{1\le i\le t'}\{b_{v_i}\}$.
\end{enumerate}
In any case, $b_v=\varsigma(T)$, and the result follows.
\end{proof}

\begin{corollary}
The positive integer $k$ for which a tree different from a path is $k$-metric dimensional can be computed in linear time with respect to the order of the tree.
\end{corollary}

\subsection{On the $k$-metric bases  and the $k$-metric dimension of  trees different from  paths}

Based on the fact that any tree different from a path is $\varsigma(T)$-metric dimensional, in this section we propose an algorithm  to compute the $k$-metric dimension and other one to determine a $k$-metric basis of any $k\le \varsigma(T)$. We first present a result with the value of the $k$-metric dimension of paths, already presented in \cite{Estrada-Moreno2013}. Further on, we center our attention to those trees different from paths.

\begin{proposition}\label{propPath}{\em \cite{Estrada-Moreno2013}}
Let $k\geq 3$ be an integer. For any path graph $P_n$ of order $n\geq k+1$, $$\dim_{k}(P_{n})=k+1.$$
\end{proposition}

We observe that, for instance, if $P_n$ is a path of order $n$ and the two leaves of $P_n$ belong to a set $S\subseteq V(P_n)$ of cardinality $k+1$, then $S$ is a $k$-metric basis of $P_n$.

We now present a function for any exterior major vertex $w\in \mathcal{M}(T)$, shown in \cite{Estrada-Moreno2013}, that allow us to compute the $k$-metric dimension of any $k\le \varsigma(T)$. Notice that this function uses the concepts already defined at the beginning of the Section \ref{k-metric-dimension-tree}. Given an integer $k\le \varsigma(T)$,

\[
I_k(w_i)=\left\{ \begin{array}{ll}
\left(ter(w_i)-1\right)\left(k-l(w_i)\right)+l(w_i), & \mbox{if } l(w_i)\le\lfloor\frac{k}{2}\rfloor,\\
& \\
\left(ter(w_i)-1\right)\lceil\frac{k}{2}\rceil+\lfloor\frac{k}{2}\rfloor, & \mbox{otherwise.}
\end{array}
\right.
\]

The following theorem presented in \cite{Estrada-Moreno2013} is the base of the two algorithms presented in this subsection.

\begin{theorem}{\em \cite{Estrada-Moreno2013}}\label{theoTreeDimR}
If $T$ is a tree which is not a path, then for any $k\in \{1,\ldots, \varsigma(T)\}$, $$\dim_{k}(T)=\sum_{w\in \mathcal{M}(T)}I_{k}(w).$$
\end{theorem}

Now we consider the problem of computing the $k$-metric dimension of a tree $T$ of order $n$, different from a path, for any $k\le \varsigma(T)$.
$$\begin{tabular}{|l|}
  \hline
  \mbox{$k$-DIMENSION TREE PROBLEM}\\
  \mbox{INSTANCE: A tree $T$ of order $n$}\\
  \mbox{PROBLEM: Compute the  $k$-metric dimension of $T$, for any $k\le \varsigma(T)$ }\\
  \hline
\end{tabular}$$

\noindent\textbf{Algorithm 2: }\\
{\bf Input:} A tree $T$ different from a path rooted in a major vertex $v$.\\
{\bf Output:} The $k$-metric dimension of $T$ for any $k\le \varsigma(T)$.
\begin{enumerate}
%Starting from a major vertex as the root of the tree $T$, we do a post-order traversal as we show in Figure \ref{figTreeNotation} for computing $r$-metric dimension of $T$. Each visited vertex $v$ store a pair $(a,b)$ in the following way:
\item For any vertex $u\in V(T)$ visited by post-order traversal as shown in Figure \ref{figTreeNotation}, assign a pair $(a_u,b_u)$ in the following way:
\begin{enumerate}
\item If $u$ does not have any child ($u$ is a leaf), then $a_u=1$ and $b_u=\infty$.
\item If $u$ has only one child ($u$ has degree $2$), then $a_u=a_{u'}+1$ and $b_u=b_{u'}$, where the pair $(a_{u'},b_{u'})$ was assigned to the child vertex of $u$. Note that $a_{u'}$ can be $\infty$, in which case $a_u=\infty$.
\item If $u$ has at least two children ($u$ is a major vertex), then $a_u=\infty$. Let $a_{min}$ be the minimum value among all $a_{u_i}$'s in the pairs $(a_{u_i},b_{u_i})$ assigned to the children of $u$, let $c_{u}$ be the number of labels $a_{u_i}$ different from $\infty$, and let $s_u$ be the sum of all $b_{u_i}\ne\infty$. If $c_u\le 1$, then $b_u=s_u$. If $c_u\ge 2$ and $a_{min}\le\lfloor\frac{r}{2}\rfloor$, then $\displaystyle b_u=a_{min}+(c_u-1)(r-a_{min})+s_u$. If $c_u\ge 2$ and $a_{min}>\lfloor\frac{r}{2}\rfloor$, then $\displaystyle b_u=\left\lfloor\frac{r}{2}\right\rfloor+(c_u-1)\left\lceil\frac{r}{2}\right\rceil+s_u$.
\end{enumerate}
\item The $k$-metric dimension of $T$ is $b_v$.
\end{enumerate}

Figure \ref{figTreeDimension} shows an example of a run of Algorithm 1 for the tree shown in Figure \ref{figTreeNotation}.

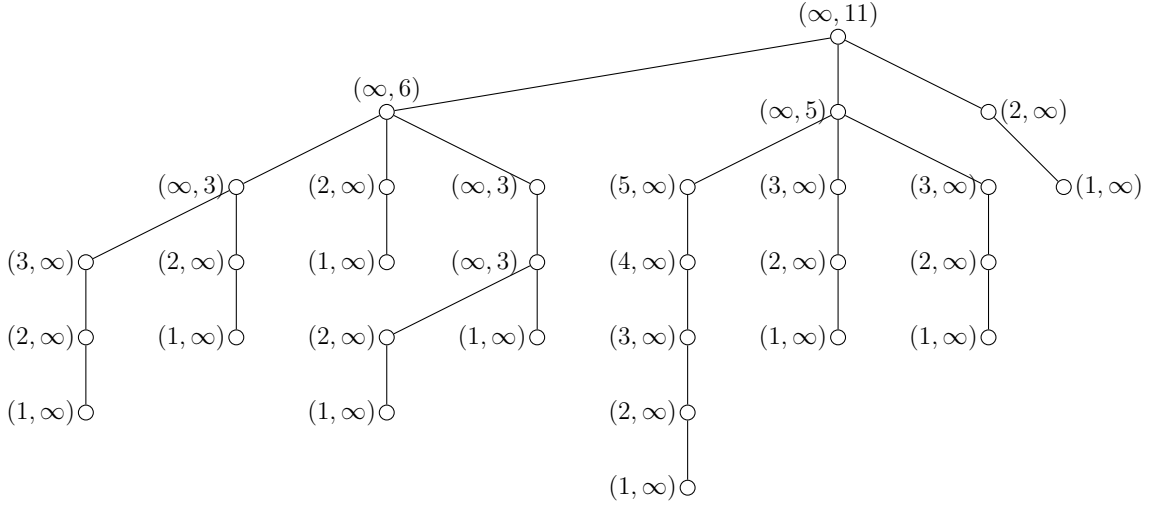
\begin{figure}[!ht]
\centering
\begin{tikzpicture}[transform shape, inner sep = .7mm]
\node [draw=black, shape=circle, fill=white] (v0) at (0,0) {};
\node [draw=black, shape=circle, fill=white] (v01) at (-6,-1) {};
\node [draw=black, shape=circle, fill=white] (v02) at (0,-1) {};
\node [draw=black, shape=circle, fill=white] (v03) at (2,-1) {};
\node [draw=black, shape=circle, fill=white] (v11) at (-8,-2) {};
\node [draw=black, shape=circle, fill=white] (v12) at (-4,-2) {};
\node [draw=black, shape=circle, fill=white] (v13) at (-6,-2) {};
\node [draw=black, shape=circle, fill=white] (v111) at (-10,-3) {};
\node [draw=black, shape=circle, fill=white] (v112) at (-8,-3) {};
\node [draw=black, shape=circle, fill=white] (v1111) at (-10,-4) {};
\node [draw=black, shape=circle, fill=white] (v11111) at (-10,-5) {};
\node [draw=black, shape=circle, fill=white] (v1121) at (-8,-4) {};
\node [draw=black, shape=circle, fill=white] (v121) at (-4,-3) {};
\node [draw=black, shape=circle, fill=white] (v1211) at (-6,-4) {};
\node [draw=black, shape=circle, fill=white] (v1212) at (-4,-4) {};
\node [draw=black, shape=circle, fill=white] (v12111) at (-6,-5) {};
\node [draw=black, shape=circle, fill=white] (v131) at (-6,-3) {};
\node [draw=black, shape=circle, fill=white] (v21) at (-2,-2) {};
\node [draw=black, shape=circle, fill=white] (v22) at (0,-2) {};
\node [draw=black, shape=circle, fill=white] (v23) at (2,-2) {};
\node [draw=black, shape=circle, fill=white] (v211) at (-2,-3) {};
\node [draw=black, shape=circle, fill=white] (v2111) at (-2,-4) {};
\node [draw=black, shape=circle, fill=white] (v21111) at (-2,-5) {};
\node [draw=black, shape=circle, fill=white] (v211111) at (-2,-6) {};
\node [draw=black, shape=circle, fill=white] (v221) at (0,-3) {};
\node [draw=black, shape=circle, fill=white] (v2211) at (0,-4) {};
\node [draw=black, shape=circle, fill=white] (v231) at (2,-3) {};
\node [draw=black, shape=circle, fill=white] (v2311) at (2,-4) {};
\node [draw=black, shape=circle, fill=white] (v31) at (3,-2) {};

\draw (v0) -- (v01) -- (v11) -- (v111) -- (v1111) -- (v11111);
\draw (v11) -- (v112) -- (v1121);
\draw (v01) -- (v12) -- (v121) -- (v1211) -- (v12111);
\draw (v01) -- (v13) -- (v131);
\draw (v121) -- (v1212);
\draw (v0) -- (v02) -- (v21) -- (v211) -- (v2111) -- (v21111) -- (v211111);
\draw (v02) -- (v22) -- (v221) -- (v2211);
\draw (v02) -- (v23) -- (v231) -- (v2311);
\draw (v0) -- (v03) -- (v31);

\node [scale=.8] at ([xshift=-.6 cm]v11111) {$(1,\infty)$};
\node [scale=.8] at ([xshift=-.6 cm]v1111) {$(2,\infty)$};
\node [scale=.8] at ([xshift=-.6 cm]v111) {$(3,\infty)$};
\node [scale=.8] at ([xshift=-.6 cm]v1121) {$(1,\infty)$};
\node [scale=.8] at ([xshift=-.6 cm]v112) {$(2,\infty)$};
\node [scale=.8] at ([xshift=-.6 cm]v11) {$(\infty,3)$};
\node [scale=.8] at ([xshift=-.6 cm]v131) {$(1,\infty)$};
\node [scale=.8] at ([xshift=-.6 cm]v13) {$(2,\infty)$};
\node [scale=.8] at ([xshift=-.6 cm]v12111) {$(1,\infty)$};
\node [scale=.8] at ([xshift=-.6 cm]v1211) {$(2,\infty)$};
\node [scale=.8] at ([xshift=-.6 cm]v1212) {$(1,\infty)$};
\node [scale=.8] at ([xshift=-.7 cm]v121) {$(\infty,3)$};
\node [scale=.8] at ([xshift=-.7 cm]v12) {$(\infty,3)$};
\node [scale=.8] at ([yshift=.3 cm]v01) {$(\infty,6)$};
\node [scale=.8] at ([xshift=-.6 cm]v211111) {$(1,\infty)$};
\node [scale=.8] at ([xshift=-.6 cm]v21111) {$(2,\infty)$};
\node [scale=.8] at ([xshift=-.6 cm]v2111) {$(3,\infty)$};
\node [scale=.8] at ([xshift=-.6 cm]v211) {$(4,\infty)$};
\node [scale=.8] at ([xshift=-.6 cm]v21) {$(5,\infty)$};
\node [scale=.8] at ([xshift=-.6 cm]v2211) {$(1,\infty)$};
\node [scale=.8] at ([xshift=-.6 cm]v221) {$(2,\infty)$};
\node [scale=.8] at ([xshift=-.6 cm]v22) {$(3,\infty)$};
\node [scale=.8] at ([xshift=-.6 cm]v2311) {$(1,\infty)$};
\node [scale=.8] at ([xshift=-.6 cm]v231) {$(2,\infty)$};
\node [scale=.8] at ([xshift=-.6 cm]v23) {$(3,\infty)$};
\node [scale=.8] at ([xshift=-.6 cm]v02) {$(\infty,5)$};
\node [scale=.8] at ([xshift=.6 cm]v31) {$(1,\infty)$};
\node [scale=.8] at ([xshift=.6 cm]v03) {$(2,\infty)$};
\node [scale=.8] at ([yshift=.3 cm]v0) {$(\infty,11)$};
\end{tikzpicture}
\caption{Algorithm 2 yields that $3$-metric dimension of this tree is $11$.}
\label{figTreeDimension}
\end{figure}

\begin{remark}
Let $T$ be a tree different from a path. Algorithm 2 computes the $k$-metric dimension of $T$ for any $k\le \varsigma(T)$.
\end{remark}

\begin{proof}
Let $v$ be the major vertex taken as a root of the tree $T$ different from a path and let $(a,b)$ be the pair stored in $v$ once Algorithm 2 has been executed. We shall show that $\displaystyle b_v=\sum_{v'\in\mathcal{M}(T)}I_k(v')$. Since $v$ is a major vertex, it has at least three children. Let $t\ge 3$ be the number of children of $v$ and let $S_1,\ldots,S_t$ be the subtrees whose roots are the children $v_1,\ldots,v_t$ of $v$, respectively. We differentiate two cases:
\begin{enumerate}
\item There exist at least two subtrees that are paths. In this case $v\in\mathcal{M}(T)$. Let $S_1,\ldots,S_{t'}$ be the subtrees that are paths, where $2\le t'\le t$. In this case, once executed Algorithm 2, each root $v_i$ of $S_i$, $1\le i\le t'$, stores the pair $(a_{v_i},\infty)$, where $a_{v_i}$ is the number of vertices of $S'_i$. Note that in this case $\ter(v)=c_v=t'\ge 2$ and $l(v)=a_{min}$. If $a_{min}\le\lfloor\frac{k}{2}\rfloor$, then $I_k(v)=a_{min}+(c_v-1)(k-a_{min})$. Otherwise, $I_k(v)=\lfloor\frac{k}{2}\rfloor+(c_v-1)\lceil\frac{k}{2}\rceil$. If $t'=t$, then $v$ is the only exterior major vertex and $s_v=0$. As a consequence, Algorithm 2 has assigned to $v$ the pair $(a_v,b_v)=(\infty,I_k(v))$. Assume that $t'<t$. Thus, there exists at least one subtree that is not path. Let $S_{t'+1},\ldots,S_t$ be the subtrees that are not paths. We consider the root $v_i$ of $S_i$, $t'+1\le i\le t$. If $v_i$ is a major vertex, then we take the vertex $v_i'=v_i$. Otherwise, $v_i'$ is the first descendant of $v_i$ that is a major vertex. In this case, Algorithm 2 recursively assigns to $v_i'$ the pair $(\infty, b_{v_i'})$, where $\displaystyle b_{v_i'}=\sum_{v'\in\mathcal{M}(T)\cap V(S_i)}I_k(v')$. In both cases, $b_{v_i}=b_{v_i'}$, where $(\infty, b_{v_i})$ is the pair assigned to $v_i$ by Algorithm 2. Hence, $\displaystyle s_v=\sum_{v'\in\mathcal{M}(T)-\{v\}}I_k(v')$. Therefore, the execution of Algorithm 2 assigns to $v$ the pair $\displaystyle (a_v,b_v)=\left(\infty, I_k(v)+\sum_{v'\in\mathcal{M}(T)-\{v\}}I_k(v')\right)=\left(\infty, \sum_{v'\in\mathcal{M}(T)}I_k(v')\right)$.
\item There exists at most one subtree that is a path. In this case $v\notin\mathcal{M}(T)$ and $c_v\le 1$. Let $S_1,\ldots,S_t'$ be the subtrees that are not paths, where $1\le t'\le t$. For each root $v_i$ of $S_i$, $1\le i\le t'$, if $v_i$ is a major vertex, then we take the vertex $v_i'=v_i$. Otherwise, $v_i'$ is the first descendant of $v_i$ that is a major vertex. In this case, Algorithm 2 recursively assigns to $v_i'$ the pair $(\infty, b_{v_i'})$, where $\displaystyle b_{v_i'}=\sum_{v'\in\mathcal{M}(T)\cap V(S_i)}I_k(v')$. In both cases, $b_{v_i}=b_{v_i'}$, where $(\infty, b_{v_i})$ is the pair stored in $v_i$ by an execution of Algorithm 2. Hence, $\displaystyle s_v=\sum_{v'\in\mathcal{M}(T)}I_k(v')$. Note in this case, at most one of all the $a_{v_i}$'s belonging to the pairs $(a_{v_i},b_{v_i})$ assigned to the children of $v$ is different from infinity. As a consequence, $c_v\le 1$. Therefore, Algorithm 2 assigns to $v$ the pair $\displaystyle (a_v,b_v)=\left(\infty, \sum_{v'\in\mathcal{M}(T)}I_k(v')\right)$.
\end{enumerate}
In any case, $\displaystyle b_v=\sum_{v'\in\mathcal{M}(T)}I_k(v')$, and the result follows.
\end{proof}

\begin{corollary}
The $k$-metric dimension of any tree different from a path, for any $k\le \varsigma(T)$, can be computed in linear time with respect to the order of $T$.
\end{corollary}

Now we consider the problem of finding a $k$-metric basis of a tree different from a path for any $k\le \varsigma(T)$. To this end, we present an algorithm quite similar to Algorithm 2, which is based on the $k$-metric basis of $T$ proposed in the proof of Theorem \ref{theoTreeDimR}.

$$\begin{tabular}{|l|}
  \hline
  \mbox{$k$-METRIC BASIS TREE PROBLEM}\\
  \mbox{INSTANCE: A tree $T$ of order $n$ different from a path}\\
  \mbox{PROBLEM: Find a $k$-metric basis of $T$, for any $k\le \varsigma(T)$}\\
  \hline
\end{tabular}$$

%\textbf{Algorithm 3: }Starting from a major vertex as the root of the tree $T$, we do a post-order traversal for finding $r$-metric basis of $T$. Each visited vertex $v$ store a pair $(a,b)$ in the following way:
\noindent\textbf{Algorithm 3: }\\
{\bf Input:} A tree $T$ different from a path rooted in a major vertex $v$.\\
{\bf Output:} A $k$-metric basis of $T$ for any $k\le \varsigma(T)$.
\begin{enumerate}
\item For any vertex $u\in V(T)$ visited by post-order traversal as shown in Figure \ref{figTreeNotation}, assign a pair $(a_u,b_u)$ in the following way:
\begin{enumerate}
\item If $u$ does not have any child ($u$ is a leaf), then $a=\{u\}$ and $b=\emptyset$.
\item If $u$ has only one child ($u$ has degree $2$), then $b_u=b_{u'}$, where the pair $(a_{u'},b_{u'})$ was assigned to the child vertex of $u$. If $a_{u'}=\emptyset$, then $a_u=\emptyset$. If $a_{u'}\ne\emptyset$, then $a_u=a_{u'}\cup\{u\}$.
\item If $u$ has at least two children ($u$ is a major vertex), then $a_u=\emptyset$. Let $a_{min}$ be a set of minimum cardinality among all $a_{u_i}$ belonging to the pairs $(a_{u_i},b_{u_i})$ assigned to the children of $u$, let $c_u$ be the number of $a_{u_i}$ which are different from an empty set, and let $d_u$ be the union of all $b_{u_i}$. If $c_u\le 1$, then $b_u=d_u$. If $c_u\ge 2$ and $|a_{min}|\le\lfloor\frac{k}{2}\rfloor$, then we remove elements of each $a_{u_i}\ne a_{min}$ until its cardinality is $k-|a_{min}|$. If $c_u\ge 2$ and $|a_{min}|>\lfloor\frac{k}{2}\rfloor$, then we remove elements of each $a_{u_i}\ne a_{min}$ until its cardinality is $\lceil\frac{k}{2}\rceil$, and we remove elements of $a_{min}$ until its cardinality is $\lfloor\frac{k}{2}\rfloor$. Then $\displaystyle b_u=a_{min}\cup\left(\bigcup_{a_{u_i}\ne a_{min}}a_{u_i}\right)\cup d_u$.
\end{enumerate}
\item A $k$-metric basis of $T$ is stored in $b_v$.
\end{enumerate}

\begin{remark}
Let $T$ be a tree different from a path. Algorithm 3 finds an $k$-metric basis of $T$ for any $k\le \varsigma(T)$.
\end{remark}

\begin{proof}
Given an exterior major vertex $w\in\mathcal{M}(T)$ such that $u_1,u_2,\ldots,u_t$ are its terminal vertices and $l(w)=l(u_{min},w)$, we define the vertex set $B_k(w)$ in the following way. If $l(v)\le\lfloor\frac{k}{2}\rfloor$, then $|B_k(w)\cap (V(P(u_j,w))-\{w\})|=k-l(v)$, for any $j\ne min$, and $V(P(u_{min},w))-\{w\}\subset B_k(w)$. Otherwise, $|B_k(w)\cap (V(P(u_j,w))-\{w\})|=\lceil\frac{k}{2}\rceil$, for any $j\ne min$, and $|B_k(w)\cap (V(P(u_{min},w))-\{w\})|=\lfloor\frac{k}{2}\rfloor$. It was shown in \cite{Estrada-Moreno2013}, that $\bigcup_{w\in\mathcal{M}(T)}B_k(w)$ is a $k$-metric basis of $T$. Let $v$ be the major vertex taken as a root of the tree $T$ different from a path, and let $(a_v,b_v)$ be the pair assigned to $v$ once executed Algorithm 3. We show that the vertex set $b_v=\bigcup_{w\in\mathcal{M}(T)}B_k(w)$. Since $v$ is a major vertex, it has at least three children. Let $t\ge 3$ be the number of children of $v$ and let $S_1,\ldots,S_t$ be the subtrees whose roots are the children $v_1,\ldots,v_t$ of $v$, respectively. We differentiate two cases:
\begin{enumerate}
\item There exist at least two subtrees that are paths. In this case $v\in\mathcal{M}(T)$. Let $S_1,\ldots,S_{t'}$ be the subtrees that are paths, where $2\le t'\le t$. Hence, Algorithm 3 assigns to each root $v_i$ of $S_i$, $1\le i\le t'$, the pair $(a_{v_i},\emptyset)$, where $a_{v_i}=V(S_i)$. Note that in that situation $\ter(v)=c_v=t'\ge 2$ and $l(v)=|a_{min}|$. If $t'=t$, then $v$ is the only exterior major vertex and $d_v=\emptyset$. As a consequence, Algorithm 3 assigns to $v$ the pair $(\emptyset, B_k(v))$. Assume now that $t'<t$. Thus, there exists at least one subtree that is not a path. Let $S_{t'+1},\ldots,S_t$ be the subtrees that are not paths. For each root $v_i$ of $S_i$, $t'+1\le i\le t$, if $v_i$ is a major vertex, then we take the vertex $v_i'=v_i$. Otherwise, $v_i'$ is the first descendant of $v_i$ that is a major vertex. Hence, Algorithm 3 recursively stores in $v_i'$ the pair $(\infty, b_{v_i'})$, where $\displaystyle b_{v_i'}=\bigcup_{v'\in\mathcal{M}(T)\cap V(S_i)}B_k(v')$. In both cases, $b_{v_i}=b_{v_i'}$, where $(\infty, b_{v_i})$ is the pair stored in $v_i$ by Algorithm 3. Hence, $\displaystyle d_v=\bigcup_{v'\in\mathcal{M}(T)-\{v\}}B_k(v')$. Therefore, Algorithm 3 assigns to $v$ the pair $\displaystyle (a_v,b_v)=\left(\emptyset, B_k(v)\cup\bigcup_{v'\in\mathcal{M}(T)-\{v\}}B_k(v')\right)=\left(\emptyset, \bigcup_{v'\in\mathcal{M}(T)}B_k(v')\right)$.
\item There exists at most one subtree that is a path. In this case $v\notin\mathcal{M}(T)$ and $c_v\le 1$. Let $S_1,\ldots,S_t'$ be the subtrees that are not paths, where $1\le t'\le t$. For each root $v_i$ of $S_i$, $1\le i\le t'$, if $v_i$ is a major vertex, then we take the vertex $v_i'=v_i$. Otherwise, $v_i'$ is the first descendant of $v_i$ that is a major vertex. Hence, Algorithm 3 has recursively assigned to $v_i'$ the pair $(\infty, b_{v_i'})$, where $\displaystyle b_{v_i'}=\bigcup_{v'\in\mathcal{M}(T)\cap V(S_i)}B_k(v')$. Again, $b_{v_i}=b_{v_i'}$, where $(\infty, b_{v_i})$ is the pair stored in $v_i$ by Algorithm 3. Thus, $\displaystyle d_v=\bigcup_{v'\in\mathcal{M}(T)}B_k(v')$. Note in such case, at most one of all possible $a_{v_i}$'s belonging to the pairs $(a_{v_i},b_{v_i})$ assigned to the children of $v$ is different from infinity. As a consequence, $c_v\le 1$. Therefore, Algorithm 3 assigns to $v$ the pair $\displaystyle (a_v,b_v)=\left(\emptyset,\bigcup_{v'\in\mathcal{M}(T)}B_k(v')\right)$.
\end{enumerate}
In any case, $\displaystyle b_v=\bigcup_{v'\in\mathcal{M}(T)}B_k(v')$, and the result follows.
\end{proof}

\begin{corollary}
A $k$-metric basis of any tree different from a path, for any $k\le \varsigma(T)$, can be computed in linear time with respect to the order of $T$.
\end{corollary}

\end{document}